%% file: rev_20230128.tex
\documentclass[final]{siamltex}

\pdfoutput=1

\usepackage{epsfig}
\usepackage[english]{babel}
\usepackage{amssymb,amsmath,amsfonts}
\usepackage{mathrsfs}
\usepackage{mdwlist}
\usepackage{graphicx,enumerate,url,hyperref,enumitem,multirow,multicol}
\usepackage[usenames,dvipsnames,svgnames]{xcolor}
\hypersetup{pdfauthor={Feng Xue},
    pdftitle={DRS fails to be a proximal mapping},
    colorlinks=true,
    linkcolor=blue,
    citecolor=red,
    urlcolor=blue,
}

\usepackage{algorithm}
\usepackage{algcompatible}
\algnewcommand\algorithmicreturn{\textbf{return}}
\algnewcommand\RETURN{\State \algorithmicreturn}%
\usepackage[outercaption]{sidecap}

\usepackage{tikz}
\usepackage{bm}

\usepackage{pgfplots}
\pgfplotsset{
  tick label style={font=\footnotesize},
  label style={font=\footnotesize},
  legend style={font=\footnotesize}
}

\usepackage{booktabs}
\usepackage{caption,subcaption}

\usepackage{todonotes}
\input{info}

\usepackage[normalem]{ulem} % for sout, can remove for final version

\newif\ifcompilePGFfigs
\compilePGFfigstrue

\graphicspath{{figs/}}

\title{On Douglas-Rachford splitting that generally fails to be a proximal mapping: A degenerate proximal point analysis}

\author{Feng Xue\thanks{National key laboratory, Beijing,  China (\tt{fxue@link.cuhk.edu.hk}).}  }
\date{\today}

\begin{document}

\maketitle

\begin{abstract}
Based on a degenerate proximal point analysis, we show that the Douglas-Rachford splitting can be reduced to a well-defined resolvent, but generally fails to be a proximal mapping. This extends the recent result of [Bauschke, Schaad and Wang. Math. Program., 168 (2018), pp. 55--61] to more general setting. The related concepts and consequences are also discussed. In particular, the results regarding the maximal and cyclic monotonicity are instrumental for analyzing many operator splitting algorithms.
\end{abstract}

% REQUIRED

% REQUIRED

\section{Introduction}
It was shown in a recent work of \cite{bau_fail} that the Douglas-Rachford splitting (DRS) fails to be a proximal mapping in a class of symmetric linear relations that are maximally monotone. We in this short note extend this result to general setting under a degenerate proximal point framework recently established in \cite{bredies_preprint}.

This work also extends the results of an early seminal work \cite{drs_1992}, by (i) giving an   explicit form of the resolvent corresponding to the DRS, compared to the implicit expression given in \cite[Sect. 4]{drs_1992}; (ii) proving its maximality of the associated monotone operator on its own right, without resorting to the full domain of the resolvent, as \cite[Theorem 4]{drs_1992} did.

The rest of this paper is organized as follows. In Sect. \ref{sec_aux}, we develop some auxiliary results on degenerate proximal point algorithm,  maximal monotonicity and cyclic monotonicity. Sect. \ref{sec_main} contains the main result, which is further discussed in Sect. \ref{sec_dis}.

Our definitions and notations are standard and follow  largely, e.g.,  \cite{rtr_book,plc_book}.

\section{Auxiliary results} \label{sec_aux}
\subsection{Reformulation of DRS as a degenerate proximal point algorithm}
For solving a monotone inclusion:
\be \label{problem}
\text{find\ } x\in \cH, \quad \text{such that\ } 0 \in (A+B)x,
\ee
where $A: \cH \mapsto 2^\cH$ and $B: \cH \mapsto 2^\cH$ are (set-valued) maximally monotone on $\cH$, the standard DRS reads as \cite{lions}
\be \label{drs}
z^{k+1} := z^k -   J_{\tau B}  (z^k) +  J_{\tau A}  (2   J_{\tau B}( z^k)   -   z^k \big) .
\ee

The following lemma shows that the DRS scheme \eqref{drs} exactly fits into a standard form of a proximal point algorithm (PPA)  \cite{bredies_preprint}:
\be \label{ppa}
b^{k+1} : \in (\cA  +\cQ)^{-1} \cQ b^k.
\ee

\begin{lemma} \label{l_ppa}
Given the standard DRS scheme \eqref{drs}, the following hold.

{\rm (i)} The scheme \eqref{drs} is equivalent to
\be \label{drs_4}
\left\lfloor \begin{array}{lll}
 u^{k+1} & := &  J_{\frac{1}{\tau} B^{-1} } \big(
 \frac{1}{\tau} z^k \big) ,  \\
 s^{k+1} & := &   J_{\frac{1}{\tau} A^{-1} } \big(\frac{1}{\tau} z^k -2u^{k+1} \big),  \\
 z^{k+1} & :=  & z^k - \tau ( s^{k+1} +  u^{k+1} ) .
 \end{array} \right.
\ee

{\rm (ii)} \eqref{drs_4} fits into the PPA form \eqref{ppa}  with $b \in \cH \times \cH \times \cH$, $\cA: \cH \times \cH \times \cH \mapsto 2^\cH \times 2^\cH \times \cH$ and $\cQ: \cH \times \cH \times \cH \mapsto \{0\} \times \{0\} \times \cH$ given as
\[
b = \begin{bmatrix}
 u  \\ s \\ z \end{bmatrix},\  
 \cA = \begin{bmatrix}
B^{-1}  & -\tau I   & - I  \\
\tau I   & A^{-1}  &   -I \\
I   & I  &  0
\end{bmatrix},\   
\cQ =    \begin{bmatrix}
 0  &  0 &  0 \\  0 &  0    &  0 \\
  0 &   0  & \frac{1}{\tau} I 
\end{bmatrix}  .
\]
\end{lemma}

\begin{proof}
(i) Letting $x^{k+1}  :=  J_{\tau B} (z^k)$, and 
$ w^{k+1}   :=   J_{\tau A} (2x^{k+1} - z^k)$, the standard DRS \eqref{drs} can be developed as \cite[Eq.(21)]{vanden_2018}
\[
\left\lfloor \begin{array}{lll}
 x^{k+1} & := &  J_{\tau B } ( z^k), \\
  w^{k+1} &  := &   J_{\tau A} (2 x^{k+1} -  z^k), \\
 z^{k+1} & := &  z^k +  ( w^{k+1}  -   x^{k+1}),
\end{array} \right.
\]
which is equivalent to the inclusion form:
\[
\left\lfloor \begin{array}{lll}
B x^{k+1}   & \owns & \frac{1}{\tau}(z^{k} -  x^{k+1} ) ,  \\
A w^{k+1} & \owns & \frac{1}{\tau} (   2 x^{k+1} -  z^{k} -w^{k+1} ) , \\
z^{k+1} & =  &  z^k +  w^{k+1} -  x^{k+1}  .
 \end{array} \right.
\]
Substituting $u^{k+1}=\frac{1}{\tau} ( z^k- x^{k+1})  $ and $s^{k+1} = \frac{1}{\tau} ( 2 x^{k+1} -  z^{k} - w^{k+1}) $ into the above scheme, and removing $(x,w)$, we obtain
\be \label{w1}
\left\lfloor \begin{array}{lll}
B^{-1} u^{k+1} & \owns &  z^k-\tau u^{k+1}, \\
A^{-1} s^{k+1} & \owns &  z^{k} - 2  \tau u^{k+1}
- \tau s^{k+1} , \\ 
 z^{k+1} & =  &  z^k - \tau ( u^{k+1} + s^{k+1}),
 \end{array} \right.
\ee
which is equivalent to  \eqref{drs_4}.

\vskip.1cm
(ii) To express \eqref{w1} as the PPA form \eqref{ppa}, substituting the $z$-step into $ u$- and $s$-steps yields
\[
\left\{ \begin{array}{lll}
  0 &\in & B^{-1} u^{k+1}  +   \tau  u^{k+1} - z^k  
=  B^{-1} u^{k+1}  -  \tau s^{k+1} -   z^{k+1}, \\
  0 & \in &  A^{-1} s^{k+1}  + \tau  s^{k+1} - z^k +2\tau  u^{k+1}   =   A^{-1} s^{k+1}  -  z^{k+1} +  \tau   u^{k+1}.   
\end{array} \right. 
\]
Thus, \eqref{w1} can be rewritten as
\[
\begin{bmatrix}
  0 \\ 0 \\  0 \end{bmatrix} \in \begin{bmatrix}
B^{-1}   & -\tau I  & -I  \\
\tau I  &A^{-1}   &   -I \\
I    & I  &  0 \end{bmatrix}  
\begin{bmatrix}
u^{k+1} \\ s^{k+1} \\ z^{k+1}  
\end{bmatrix} +   \begin{bmatrix}
 0  &  0 &  0 \\  0 &  0    &  0 \\
  0 &  0  & \frac{1}{\tau}  I 
\end{bmatrix}  \begin{bmatrix}  u^{k+1} - u^k \\ 
s^{k+1} - s^k \\ z^{k+1}  - z^k
\end{bmatrix} .
\]
from which follow the notations of $b$, $\cA$ and $\cQ$.  \hfill
\end{proof}

Notice that the corresponding metric $\cQ$ in Lemma \ref{l_ppa}-(ii) is merely positive {\it semi-}definite, rather than {\it strictly} positive definite---a standard metric setting in classical PPA \cite{ppa_guler,rtr_1976,fxue_gopt}. This phenomenon is referred to as {\it degeneracy}, which  has recently been noticed and systematically studied in \cite{fxue_1,bredies_preprint}. The degenerate metric $\cQ$ here implies that the variables $(u,s)$ lying in $\ker\cQ$ are auxiliary and redundant that do not really take part in the iterative process \eqref{drs_4}. To see this, substituting $u$ and $s$-steps into $z$-step of \eqref{drs_4}, \eqref{drs} is exactly recovered, which is an iterative process of the only active variable $z^k$. It indicates that  $(u^k,s^k)$ are merely intermediate results of \eqref{drs_4}. Though $b\in \cH^3$ in the apparent PPA representation, the true dimension of the scheme  \eqref{drs_4} remains $\cH$. 

\vskip .1cm
Lemma \ref{l_ppa} performs the `{\it size expansion}' from $\cH$ of  \eqref{drs} to $\cH^3$ of \eqref{drs_4}\footnote{\eqref{drs_4} only increases the apparent size of \eqref{drs}, while keeping the actual dimension unchanged. Hence, we use `{\it size expansion}' rather than {\it dimension expansion}.}. The next result will reduce the apparent size $\cH^3$ of \eqref{drs_4} back to the original size $\cH$, based on a recent result of \cite[Theorem 2.13]{bredies_preprint}. After the `expansion+reduction' steps, we obtain an equivalent resolvent to DRS \eqref{drs} as follows.
\begin{proposition} \label{p_res}
The DRS scheme \eqref{drs} can be expressed as
\be \label{drs_eq}
v^{k+1} = (\cI+   \cK \cL^{-1} \cK^\top)^{-1} v^k, 
\ee
where $\cL = \begin{bmatrix}
B^{-1}  & -\tau I \\   \tau I  & A^{-1}  \end{bmatrix}: \cH \times \cH \mapsto 2^\cH \times 2^\cH$,
$\cK = \sqrt{\tau} \begin{bmatrix}  I & I \end{bmatrix}: \cH \times \cH \mapsto \cH$. Here, the variable $v \in \cH$ is linked to $x$ in \eqref{problem} and $z$ in \eqref{drs} via  $x^k = J_{\tau B} (z^k) =  J_{\tau B} ( \sqrt{\tau}  v^k) $.
\end{proposition}
\begin{proof}
Lemma \ref{l_ppa}-(ii) presents $\cA$ and $\cQ$ of PPA corresponding to the DRS scheme \eqref{drs}. The metric $\cQ$ can be decomposed as
\[
\cQ = \cD \cD^\top = \begin{bmatrix}
0 \\ 0 \\ \frac{1}{\sqrt{\tau}} I \end{bmatrix} 
\begin{bmatrix}
0 & 0 & \frac{1}{\sqrt{\tau}} I \end{bmatrix},
\]
where $\cD: \cH \mapsto \{0\} \times \{0\} \times \cH$. 
 Let $v^k := \cD^\top b^k =   \frac{1}{\sqrt{\tau}} z^k$.  Applying \cite[Theorem 2.13]{bredies_preprint} to \eqref{ppa} yields the reduced PPA: 
\begin{eqnarray}
v^{k+1} &=& \cD^\top (\cA +\cD\cD^\top)^{-1} \cD v^k
\quad \text{(multiplying $D^\top$ on both sides of \eqref{ppa})}
\nonumber \\
&=& \big(\cI + (\cD^\top\cA^{-1} \cD)^{-1}\big)^{-1} v^k.
\quad \text{(by \cite[Theorem 2.13]{bredies_preprint})}
\nonumber
\end{eqnarray}
Now, we need to solve $\cD^\top\cA^{-1} \cD$. Denote $\cR: = \cA^{-1} \cD = \begin{bmatrix}
\cR_1 \\  \cR_2 \end{bmatrix}$, then, $\cD^\top  \cA^{-1} \cD = \frac{1}{\sqrt{\tau}} \cR_2$. To find $\cR_2$,  we
rewrite $\cA$ in Lemma \ref{l_ppa}-(ii) as $\cA = \begin{bmatrix}
\cL & -\cK^\top \\ \cK & 0 \end{bmatrix}$, where
 $\cL = \begin{bmatrix}
B^{-1}  & -\tau I \\ 
\tau I  & A^{-1}  \end{bmatrix}$,
$\cK = \begin{bmatrix}
I & I \end{bmatrix}$. Then,   $\cR$ satisfies
\[
\cA \cR = \cD \Longrightarrow 
  \begin{bmatrix}
\cL & -\cK^\top \\ \cK & 0 \end{bmatrix}
 \begin{bmatrix}
\cR_1 \\ \cR_2  \end{bmatrix} = 
\begin{bmatrix}
0 \\ \frac{1}{\sqrt{\tau}} I \end{bmatrix},
\]
i.e., $\cL \cR_1 = \cK^\top \cR_2$ and $\cK \cR_1 = \frac{1}{\sqrt{\tau}} \cI$. Substituting $\cR_1 = \cL^{-1} \cK^\top \cR_2$ into the second equation, we obtain  $ \cR_2=  (\cK  \cL^{-1} \cK^\top)^{-1} \circ (\frac{1}{\sqrt{\tau}}\cI)$.  Thus,
\[
(\cD^\top\cA^{-1}\cD)^{-1} =   
\big( \frac{1}{\sqrt{\tau}} \cR_2 \big)^{-1} = (\sqrt{ \tau}\cI) \circ  (\cK \cL^{-1} \cK^\top) \circ (\sqrt{ \tau}\cI)
=(\sqrt{ \tau}\cK) \circ \cL^{-1} \circ (\sqrt{\tau}\cK^\top).
\]
Merging the factor of $\sqrt{\tau}$ into $\cK$ completes the proof. \hfill 
\end{proof}

\subsection{Maximal monotonicity and cyclic monotonicity}
To investigate the monotone properties of  $\cK \cL^{-1} \cK^\top$ given in Proposition \ref{p_res}, let us first discuss the maximal and cyclic monotonicity under more general setting. 

We define the proximal mapping.
\begin{definition} {\rm \cite[Definition 1.22, Example 10.2]{rtr_book_2} } \label{def_prox}
The operator $\cT$ is a proximal mapping, if there exists a  proper, lower semi-continuous (l.s.c.) and convex function $f$, such that  $\cT = (\cI + \partial f)^{-1}$. 
\end{definition}

The following results will be used in Sect. \ref{sec_main}. 
\begin{lemma} \label{l_K}
Given a set-valued maximally monotone operator $\cS: \cH \mapsto 2^\cH$ and a linear surjective operator $\cP: \cH \mapsto \cH$, then, $\cP \cS \cP^\top$ is  maximally monotone.
\end{lemma}
\begin{proof}
(i) The monotonicity of  $\cP\cS \cP^\top$ immediately follows by \cite[Proposition 20.10]{plc_book}. 

(ii) Check the maximality, i.e. does 
\[
\langle x-y | u- v \rangle \ge 0, \quad
\forall (y,v) \in \gra (\cP \cS \cP^\top),
\]
imply $u \in \cP \cS \cP^\top x$?

Since $\cP$ is surjective, $u$ can always be expressed as $u=\cP u'$ for some $u'$. Then, the above inequality becomes
\[
\langle x-y | \cP u'- \cP v' \rangle  
= \langle \cP^\top x- \cP^\top y |  u'-   v'  \rangle \ge 0, \quad
\forall (y, v') \in \gra ( \cS \cP^\top).
\]
The maximality of $\cS$ yields  $u'\in \cS \cP^\top x$, from which we conclude that  $u=\cK u' \in \cP \cS \cP^\top x$.
\hfill 
\end{proof}

\begin{lemma} \label{l_cyc}
Given a set-valued monotone operator $\cS: \cH\mapsto 2^\cH$, $\cS$ is maximally cyclically  monotone, if and only if $\cS^{-1}$ is maximally cyclically  monotone.
\end{lemma}

\begin{proof}
By \cite[Proposition 20.22]{plc_book}, $\cS$ is  maximally monotone, if and only if $\cS^{-1}$ is maximally monotone.

Let us now check the cyclic monotonicity. First, suppose that $\cS$ is cyclically monotone. By \cite[Definition 22.13]{plc_book}, the cyclic monotonicity implies
\be \label{a}
\sum_{i=1}^n \langle x_{i+1} -x_i | u_i \rangle \le 0,\quad 
\forall u_i \in \cS x_i, \ \forall n \ge 2, 
\ee
where it is assumed that $x_{n+1} := x_1$.  \eqref{a} can also be written as $\sum_{i=1}^n \langle  u_{i-1} -   u_i | x_i \rangle \le 0$, where $u_0 := u_n$.  Since $x_i \in \cS^{-1} u_i$, this shows that $\cS^{-1}$ is also $n$-cyclic for any integer $n \ge 2$, just with reverse cyclic order of $u_i$. Likewise  the converse is also true.
 \hfill  
\end{proof}

\vskip.1cm
Observe that  $\cL$ defined in Proposition \ref{p_res} in spirit has the following structure: 
\be \label{L}
\cS = \begin{bmatrix}
A  & - C^\top \\ C  & B \end{bmatrix}: \cH_1 \times \cH_2
\mapsto 2^{\cH_1} \times 2^{\cH_2}.
\ee 
Now we discuss the properties of $\cS$ under 
\begin{assumption} \label{assume}
\begin{itemize}
\item[\rm (i)]  $A:\cH_1\mapsto 2^{\cH_1}$ and $B:\cH_2\mapsto 2^{\cH_2}$ are  monotone operators;
\item[\rm (ii)] $C:\cH_1 \mapsto \cH_2$ is a linear operator. 
\end{itemize}
\end{assumption}
This {\it monotone}+{\it skew} type of $\cS$:
\[
\cS =  \underbrace{ \begin{bmatrix}
A  & 0 \\ 0  & B \end{bmatrix}}_\text{monotone}
+ \underbrace{ \begin{bmatrix}
0  & - C^\top \\ C  & 0 \end{bmatrix} }_\text{skew}
\]
is often encountered in many operator splitting algorithms, e.g.,  \cite{arias_2011,bredies_2017,fxue_gopt,vanden_2014}. 

\begin{lemma} \label{l_max}
Under Assumption \ref{assume}, the operator $\cS$ in \eqref{L} is maximally monotone,  if $A$ and $B$ are maximally monotone.
\end{lemma}
\begin{proof}
(i) It is easy to verify the monotonicity of $\cS$. 

(ii) Regarding the maximality, for every fixed $(x,u)$, let us 
examine if the following statement holds:
\[
\forall (y,w) \in \gra \cS,  \quad 
\langle x-y| u -w \rangle  \ge 0
\Longrightarrow u\in \cS x.
\]
Take $x=(x_1,x_2) \in \cH_1 \times \cH_2$ and $u=(u_1,u_2)\in \cH_1 \times \cH_2$. For any fixed $y=(y_1,y_2) \in \cH_1 \times \cH_2$, pick arbitrary $v=(v_1,v_2) \in Ay_1 \times By_2$, and correspondingly, let $w=(w_1,w_2) = (v_1-C^\top y_2, Cy_1 +v_2)$, such that $w\in \cS y$. Then, the above inequality becomes
\be \label{qq}
\langle x-y| u- w \rangle = \langle x_1-y_1| 
u_1-(v_1- C^\top y_2) \rangle +  \langle x_2-y_2| 
u_2-(C y_1+ v_2) \rangle \ge 0.
\ee
Since \eqref{qq} holds for any $y=(y_1,y_2)$, taking $y_1=x_1$ in \eqref{qq} yields
\[
\langle x_2-y_2| u_2-(C x_1+v_2) \rangle 
= \langle x_2-y_2| u_2-C x_1 - v_2  \rangle \ge 0,\ 
\forall (y_2,v_2) \in \gra B. 
\]
Due to the maximality of $B$, we conclude  $u_2-C x_1  \in Bx_2$, i.e., $u_2  \in Bx_2 + C x_1$.

Similarly, taking $y_2=x_2$ in \eqref{qq} yields
\[
 \langle x_1-y_1| u_1-(v_1-C^\top x_2) \rangle   
= \langle x_1-y_1| u_1+ C^\top x_2 - v_1  \rangle  \ge 0 ,\ 
\forall (y_1,v_1) \in \gra A. 
\]
Due to the maximality of $A$, we have   $u_1 \in Ax_1 - C^\top x_2$. Both $u_1$ and $u_2$ imply  $u \in \cS x$. The proof is completed.  \hfill 
\end{proof}

\vskip.1cm
\begin{proposition} \label{p_not}
Given $\cS$ as \eqref{L} under Assumption \ref{assume}, the following hold.
\begin{itemize}
\item[\rm (i)] $\cS$ is cyclically  monotone, if $A$ and $B$ are cyclically monotone, $C=0$;

\item[\rm (ii)] $\cS$ is maximally cyclically  monotone, if $A$ and $B$ are maximally cyclically monotone, $C=0$;

\item[\rm (iii)] $\cS$ with $C \ne 0$ is  not maximally cyclically monotone;

\item[\rm (iv)] $\cS$ with $C \ne 0$ is  not cyclically monotone, if $A$ and $B$ are maximally cyclically monotone.
\end{itemize}
\end{proposition}
\begin{proof}
(i) If $C=0$, taking $x_i=(a_i, b_i)$ for $i=1,...,n$, and picking arbitrary  $u_i =(c_i, d_i) \in Aa_i \times Bb_i = \cS x_i$ for each $i$, we develop, by the definition \eqref{L}, that
\[
\sum_{i=1}^n \langle x_{i+1} -x_i | u_i \rangle 
= \sum_{i=1}^n \langle a_{i+1} -a_i | c_i \rangle 
+ \sum_{i=1}^n \langle b_{i+1} -b_i | d_i \rangle,
\]
which is non-positive, since both terms above are non-positive due to the cyclic monotonicity of $A$ and $B$.

\vskip.2cm
(ii): in view of Lemma \ref{l_max} and (i).

\vskip.2cm
(iii) Let us prove it by contradiction. If $\cS$ is maximally cyclically monotone, then, by \cite[Theorem 22.18]{plc_book}, there exists a proper, l.s.c. and convex function $f: \cH_1 \times \cH_2 \mapsto \R\cup \{+\infty\}$ such that $\cS = \partial f$. That is, for every point $(a,b) \in \cH_1 \times \cH_2$, it holds that
\be \label{w11} 
\cS(a,b) = \begin{bmatrix}
Aa-C^\top b \\ Ca+B b   \end{bmatrix}
 = \begin{bmatrix}
\partial_a f(a,b) \\  \partial_b f(a,b)  \end{bmatrix}.
\ee
where $\partial_a$ and $\partial_b$ denote the subdifferentials w.r.t. $a$ and $b$, respectively. The first line of \eqref{w11}, i.e., $\partial_a f(a,b) = Aa-C^\top b$, implies that
$f(a,b) = \tilde{f}(a,b)-\langle a|C^\top b\rangle$, where $\partial_a \tilde{f} (a,b) = Aa$. By assumption, $\tilde{f}(a,b)$ also has to be proper, l.s.c. and convex. Let us then discuss it case-by-case.

\vskip.1cm
{\bf Case-I:} If $A$ is not maximally cyclically monotone, then, there does not exist a proper, l.s.c. and convex function $\tilde{f}(a,b)$, such that $\partial_a \tilde{f} (a,b) = Aa$.

\vskip.1cm
{\bf Case-II:} If $A$ is maximally cyclically monotone, then,  there  exists a proper, l.s.c. and convex function $f_1:\cH_1\mapsto \R\cup\{+\infty\}$, such that $\partial_a f_1(a) = Aa$. Also note $\partial_a \tilde{f} (a,b) = Aa$. By \cite[Proposition 22.19]{plc_book}, we have $\tilde{f}(a,b)=f_1(a)+g(b)$, where $g(b)$ is  viewed as a constant w.r.t. the variable $a$ (i.e., independent of $a$). Then, 
\be \label{q2}
f(a,b) = f_1(a)+g(b)-\langle a|C^\top b\rangle.
\ee
On the other hand, the second line of \eqref{w11} also requires that $\partial_b f(a,b)=Ca+Bb$. In view of  \eqref{q2}, we have $\partial_b g(b) = Bb+2Ca$, which is clearly impossible unless $C=0$,  since $g(b)$ is a function of variable $b$ only, being independent of $a$. A contradiction is reached.

\vskip.1cm
Finally, for both cases,  there is no any  proper, l.s.c. and convex function $f$, such that $\partial f = \cS$, when $C\ne 0$. Applying \cite[Theorem 22.18]{plc_book} again, $\cS$ is not maximally cyclically monotone.

\vskip.2cm
(iv) If $A$ and $B$ are maximally cyclically monotone, $C\ne 0$, $\cS$ is maximally monotone by  Lemma \ref{l_max}. On the other hand,  Proposition \ref{p_not}-(iii) claims that $\cS$ is not maximally cyclically monotone. Thus, $\cS$ is not cyclically monotone.
\hfill 
\end{proof}

\begin{remark}
Regarding Proposition \ref{p_not}-(iv), if  $A=0$ and $B=0$ (which certainly satisfies maximal cyclic monotonicity), it is easy to show a counterexample for which  $\cS$ (with $C\ne 0$) is not 3-cyclic. Denote $x_i=(a_i,b_i)$ for $i=1,2,3$, then, $u_i=\cS x_i = (-C^\top b_i, Ca_i)$. In this case, noting  $x_4=x_1$ (due to 3-cyclic), simple algebra gives
\[
\sum_{i=1}^3 \langle x_{i+1} -x_i | u_i \rangle 
=    \langle C a_1 | b_2-b_3 \rangle
+\langle C a_2 | b_3-b_1 \rangle
+\langle C a_3 | b_1-b_2 \rangle : = \xi. 
\]
Taking $x_1=(a_1,b_1)$, $x_2=(a_2,b_2)=(-C^\top b_1, Ca_1)$  and $x_3=(a_3,b_3)=(-C^\top C a_1, -CC^\top b_1)$, we obtain 
\[
\xi = \|Ca_1\|^2 +\|C^\top C a_1\|^2 + \|C^\top b_1\|^2 +\|CC^\top b_1\|^2.
\]
Since the linear operator $C \ne 0$, then, $\cH_1 \backslash \ker C \ne \emptyset$. It is easy to choose $a_1 \in \cH_1 \backslash \ker C$, such that $\|Ca_1\|^2 >0$, and hence, $\xi >0$. This  shows that $\cS$ is not 3-cyclic. 
\end{remark}

\section{Main result} \label{sec_main} 
We are now ready to present the main result.
\begin{theorem} \label{t_drs}
The DRS scheme \eqref{drs}

\begin{itemize}
\item [\rm (i)] is a well-defined resolvent;
\item [\rm (ii)]  fails to be a proximal mapping, unless $\cH = \R$.
\end{itemize}
\end{theorem}

\begin{proof}
(i) By Proposition \ref{p_res}, the DRS scheme \eqref{drs} is equivalent to \eqref{drs_eq}. Since $A$ and $B$ are assumed to be maximally monotone in the problem setting \eqref{problem},   $A^{-1}$ and $B^{-1}$ is also maximally monotone by \cite[Proposition 20.22]{plc_book}.
Thus, $\cL$ has exactly the same structure as $\cS$ given in \eqref{L}. Lemma \ref{l_max} shows that $\cL$ defined in \eqref{drs_eq} is maximally monotone, and so is $\cL^{-1}$. Since $\cK= \sqrt{\tau} \begin{bmatrix}  I & I \end{bmatrix}$ is  surjective, $\cK \cL^{-1} \cK^\top$ in \eqref{drs_eq} is  maximally monotone by 
Lemma \ref{l_K}. Thus, by the well-known Minty's theorem \cite[Theorem 21.1]{plc_book}, \eqref{drs_eq} is single-valued and well-defined everywhere.

(ii) By Definition \ref{def_prox} and \cite[Theorem 22.18]{plc_book}, a  monotone operator associated to proximal mapping has to be maximally cyclically monotone.  Proposition \ref{p_not}-(iii) shows that $\cL$ given in  \eqref{drs_eq}  is not maximally cyclically monotone, and neither are  $\cL^{-1}$ and  $\cK \cL^{-1} \cK^\top$, by Lemmas \ref{l_cyc} and \ref{l_K}.  This concludes the proof, by noting the only exception of the case of $\cH = \R$. In such a special case, $\cK \cL^{-1} \cK^\top$ is maximally cyclically monotone by \cite[Theorem 22.22]{plc_book}, and there exists a proper, l.s.c.  and convex function $f$ defined on $\R$, such that $\cK \cL^{-1} \cK^\top =\partial f$ \cite[Corollary 22.23]{plc_book}.  \hfill 
\end{proof}

\begin{remark}
We extend the work of \cite{bau_fail} in several aspects. 
\begin{itemize}
\item The discussions of \cite{bau_fail} are based on a finite-dimensional space (in order to obtain \cite[Proposition 2.6]{bau_fail}), which is extended to infinite dimensional setting here.
\item Our expositions take a completely different technical route from \cite{bau_fail}: the  central notion of \cite{bau_fail} is symmetry, while our analysis heavily relies on cyclic monotonicity.
\item \cite{bau_fail} only concludes that the DRS operator is a resolvent of some maximally monotone operator, but fails to develop its concrete form. \cite[Theorem 3.1]{bau_fail} proves the failure of DRS to be a proximal mapping by contradiction, which shows that the set of  the maximally cyclically  monotone operators $A$ and $B$ (restricted to the linear relations), such that the DRS is a proximal mapping, does not contain any elements in its interior. By contrast, we give an explicit form of the maximally monotone operator associated with the DRS, and prove the general\footnote{The word `{\it general}' means that there exists an exceptional case of $\cH = \R$.} impossibility of DRS as being  a proximal mapping in a very straightforward way.
\item Most importantly, we extend the linear relations imposed on  $A$ and $B$ in \cite{bau_fail} to more general subdifferential operators (i.e., maximally cyclically  monotone operators by \cite[Theorem 22.18]{plc_book}), and give an affirmative answer to the first open question posed in \cite[Remark 3.3]{bau_fail}.
\item Finally, note that the only exceptional case of $\cH=\R$ in Theorem \ref{t_drs}-(ii) has also been considered in \cite[Remark 3.2]{bau_fail}.
\end{itemize}
\end{remark}

\section{Extended discussions}
\label{sec_dis}
Recalling that the {\it parallel composition of $\cL$ by $\cK$} is defined as $\cL \triangleright \cK: = (\cK \cL^{-1} \cK^\top)^{-1}$ \cite{arias_parallel}, $\cK \cL^{-1} \cK^\top$ that appears in \eqref{drs_eq} can be denoted as  $(\cL \triangleright \cK)^{-1}$, and thus, \eqref{drs_eq} can be rewritten as $v^{k+1} : =J_{(\cL \triangleright \cK)^{-1}} (v^k)$. Based on the Moreau's decomposition identity (see for instance \cite[Eq.(13)]{vanden_2018}), \eqref{drs_eq} can also be expressed in terms of the resolvent of parallel composition:  $v^{k+1} : = v^k - J_{ \cL \triangleright \cK  } (v^k)$. Interested readers may refer to \cite{arias_parallel,moudafi_resolvent,fukushima} for more properties and computing methods of the resolvent of the parallel composition.  However, Proposition \ref{p_not} shows that $\cL$ given in \eqref{drs_eq} is not maximally cyclically monotone, due to its non-zero off-diagonal elements. Consequently, $\cL \triangleright \cK$ cannot be further simplified as a parallel sum of two maximally monotone operators.

One can compare Lemma  \ref{l_K} with \cite[Proposition 3.1]{fukushima}. From the deductions of \cite[Proposition 3.1]{fukushima}, we obtain that $J_{ (\cL \triangleright \cK)^{-1}}$ of \eqref{drs_eq} can also be written as
\[
J_{ (\cL \triangleright \cK)^{-1}}
= \cI - \cK (\cL +  \cK^\top\cK)^{-1} \cK^\top.
\]
\cite[Proposition 3.1]{fukushima} shows that the resolvent is defined anywhere, if $\cK^\top \cK$ is an isomorphism. Consider our context, where $\cK^\top \cK = \tau \begin{bmatrix}
\cI & \cI \\ \cI & \cI \end{bmatrix}$, which is obviously not isomorphic (since it is degenerate and non-surjective). Our Theorem \ref{t_drs}-(i) shows that this case is also a well-defined resolvent, even if $\cK^\top \cK$ is not isomorphic. This is an extension of  \cite[Proposition 3.1]{fukushima}. For this degenerate case, one may refer to the most recent works of \cite{arias_2022,arias_parallel} for detailed analysis.  
 
Many nonexpansive properties of the DRS scheme \eqref{drs} can be immediately obtained by Theorem \ref{t_drs}-(i), without using the notion of reflected resolvent, as \cite{hbs_2015,boyd_control} did. This result also greatly simplifies the analysis of the relaxed DRS: $z^{k+1} := z^k+\gamma(\ztilde^k -z^k) $, where $\ztilde^k$ is the output of the standard DRS \eqref{drs}. Using $v^k=\frac{1}{\sqrt{\tau}} z^k$, we develop
\begin{eqnarray}
v^{k+1} &:=& v^k+\gamma (\vtilde^k - v^k)
\nonumber \\
&=& v^k +\gamma \big( J_{ (\cL \triangleright \cK)^{-1}} (v^k)  -v^k \big) 
\nonumber \\
&=& \big( (1-\gamma) \cI + \gamma J_{ (\cL \triangleright \cK)^{-1}}  \big) v^k  . 
\nonumber 
\end{eqnarray}
It is easy to verify that the relaxed-DRS operator---$ (1-\gamma) \cI + \gamma J_{ (\cL \triangleright \cK)^{-1}} $---is essentially $(\gamma/2)$-averaged \cite[Definition 4.33]{plc_book}, and thus, many results of the relaxed DRS, e.g., \cite[Lemma 2.3, Lemma 2.4 and Theorem 3.1]{hbs_2015}, can be easily obtained from this averagedness.

An important consequence of the `unfortunate' result of Theorem \ref{t_drs}-(ii) is that the  convergence rate of the DRS iteration cannot be improved to that of a proximal mapping (also known as {\it proximity operator}), for instance, the stronger result of \cite[Theorem 2.2]{ppa_guler} for a standard proximal point algorithm is no longer valid for the DRS. 

\vskip.2cm
Before the end of this paper, we pose a few open problems, which appear to be of interest.
\begin{itemize}
\item [(i)] Does Lemma \ref{l_K} or \cite[Proposition 3.1]{fukushima} hold, without the surjectivity of $\cK$?

\item[(ii)] Can `{\it if}' in Proposition \ref{p_not}-(ii) be replaced by `{\it only if}'? Is this sufficient condition  also necessary?

\item[(iii)] Though the conclusion of Proposition  \ref{p_not}-(iii) has been sufficient to show Theorem \ref{t_drs},  can we improve  this result and simply claim that {\it $\cS$ with $C\ne 0$ is not cyclically monotone for general monotone operators $A$ and $B$}? In other words, is  this statement still valid, when dropping the conditions of maximality and cyclicity of $A$ and $B$?
\end{itemize}

\section{Acknowledgements}
I am gratefully indebted to the anonymous reviewer for helpful discussions, particularly related to the proofs of Lemma \ref{l_max} and Proposition \ref{p_not}. 

\section{Data availability}
There is no associated data with this manuscript.

\section{Disclosure statement}
The author declares there are no conflicts of interest regarding the publication of this paper.

%\subsubsection*{References}
%\pdfbookmark[0]{References}{references} % add a nice PDF bookmark
\bibliographystyle{siam}

\small{
\bibliography{refs}
}

\end{document}

%% file: info.tex
\newcommand{\R}{\mathbb{R}}

\newcommand{\cI}{\mathcal{I}}
\newcommand{\cR}{\mathcal{R}}
\newcommand{\cS}{\mathcal{S}}

\newcommand{\cT}{\mathcal{T}}
\newcommand{\cA}{\mathcal{A}}

\newcommand{\cP}{\mathcal{P}}

\newcommand{\cL}{\mathcal{L}}
\newcommand{\cH}{\mathcal{H}}

\newcommand{\cD}{\mathcal{D}}

\newcommand{\cQ}{\mathcal{Q}}

\newcommand{\cK}{\mathcal{K}}

\newcommand{\ztilde}{\tilde{z}}

\newcommand{\vtilde}{\tilde{v}}

\newcommand{\gra}{\mathsf{gra}}

\newcommand{\be}{\begin{equation}}
\newcommand{\ee}{\end{equation}}

\newtheorem{assumption}{Assumption}

\newtheorem{remark}{Remark}